\pdfoutput=1

\documentclass[a4paper,11pt,twoside]{article}
\usepackage{ifpdf}
\usepackage{verbatim}
\usepackage[affil-it]{authblk}
\usepackage[english]{babel}
\usepackage[a4paper,top=0.5in,bottom=0.5in,left=0.5in,right=0.5in]{geometry}
\usepackage[latin1]{inputenc}
\usepackage[T1]{fontenc}
\usepackage{Alegreya}
\usepackage{mathrsfs}
\usepackage{amssymb,amsthm,amsmath} 
\usepackage{indentfirst}
\usepackage{eufrak}
\usepackage{color}%per scrivere con diversi colori
\usepackage{graphicx}%per inserire grafici
\usepackage{xypic}
\usepackage{hyperref}
\hypersetup{
	colorlinks = true,
	linkcolor = blue,
	citecolor = red
}
\usepackage{aliascnt}
\usepackage{comment} 
\newtheorem{teorema}{Theorem}[section]

\newaliascnt{lemma}{teorema}% alias counter "<newTh>"
\newtheorem{lemma}[lemma]{Lemma}
\aliascntresetthe{lemma}

\newaliascnt{proposizione}{teorema}% alias counter "<newTh>"
\newtheorem{proposizione}[proposizione]{Proposition}
\aliascntresetthe{proposizione}

\newaliascnt{corollario}{teorema}% alias counter "<newTh>"
\newtheorem{corollario}[corollario]{Corollary}
\aliascntresetthe{corollario}

\theoremstyle{definition}
\newaliascnt{definition}{teorema}% alias counter "<newTh>"
\newtheorem{definition}[definition]{Definition}
\aliascntresetthe{definition}

\theoremstyle{remark}
\newaliascnt{osservazione}{teorema}% alias counter "<newTh>"
\newtheorem{osservazione}[osservazione]{Remark}
\aliascntresetthe{osservazione}

\newaliascnt{esempio}{teorema}% alias counter "<newTh>"

\aliascntresetthe{esempio}

\newcommand{\spec}{{\rm Spec}\,}
\newcommand{\st}{\mathscr}
\newcommand{\sh}{\mathcal}

\newcommand{\fie}{\mathbb}

\title{The integral Chow ring of the stack of 1-pointed hyperelliptic curves}

\author{Michele Pernice \\ email \href{mailto:michele.pernice@sns.it}{michele.pernice@sns.it}}
\affil{Scuola Normale Superiore, Piazza Cavalieri 7, 56126 Pisa, Italy}

\begin{document}
 
\maketitle
\begin{abstract}
In this paper we give a complete description of the integral Chow ring of the stack $\st{H}_{g,1}$ of 1-pointed hyperelliptic curves, lifting relations and generators from the Chow ring of $\st{H}_g$. We also give a geometric interpretation for the generators.
\end{abstract}

\pagenumbering{Roman}

\tableofcontents

\pagenumbering{arabic}
\section*{Introduction}

After they were introduced by Mumford in \cite{Mum},
 the rational Chow rings $ {\rm CH} (\st{M}_{g,n})_{ \fie{Q}}$ and $ {\rm CH}(\overline{ \st{M}}_{g,n})_{ \fie{Q}}$ of moduli spaces of smooth or stable curves have been the subject of extensive investigations. The integral versions $ {\rm CH} (\st{M}_{g,n})$ and $ {\rm CH}(\overline{ \st{M}}_{g,n})$ of these rings have been introduced by Edidin and Graham in \cite{EdGra},
  and are much harder to compute. The known results are for $ {\rm CH}( \st{M}_{1,1})$ \cite{EdGra},
   $ {\rm CH}( \st{M}_{2})$ \cite{Vis3}
    and $ {\rm CH}(\overline{ \st{M}}_{2})$ \cite{Lar}.

A stack whose Chow ring is particularly amenable to be studied with the equivariant techniques of \cite{EdGra} 
is the stack $ \st{H}_{g}$ of smooth hyperelliptic curves of genus $g \geq 2$, introduced in \cite{ArVis},
 where its Picard group is computed. Its Chow ring has been computed by Edidin--Fulghesu \cite{EdFul}
  for even $g$, and by Fulghesu--Viviani and Di Lorenzo for odd $g$ \cite{FulViv, DiLor}
   (the second paper introduces new ideas to fill a gap in the first). The result is the following.

\begin{teorema}[Edidin--Fulghesu, Fulghesu--Viviani, Di Lorenzo]
	If $g$ is even, then
	\[
	 {\rm CH} (\st{H}_{g}) =
	\fie{Z}[c_1, c_2]/\bigl(2(2g + 1)c_{1}, g(g - 1)c_{1}^{2} - 4g(g + 1)c_{2}\bigr)\,.
	\]
	If $g$ is odd, then
	\[
      {\rm CH} (\st{H}_{g} )=
	\fie{Z}[\tau, c_{2}, c_{3}]/
	\bigl(4(2g + 1)\tau, 8\tau^{2} - 2(g^{2} - 1)c_{2}, 2c_{3}\bigr)\,.
	\]
\end{teorema}

Here the $c_{i}$'s are Chern classes of certain natural vector bundles on $ \st{H}_{g}$ and $\tau$ is the first Chern class of a certain line bundle.

In this paper we compute the Chow ring of the stack $ \st{H}_{g,1}$ of smooth $1$-pointed hyperelliptic curves of genus $g$ for any $g \geq 2$. Our main result is as follows.

\begin{teorema}(See \hyperref[pri]{Theorem \ref{pri}})
	\begin{enumerate}
		
		\item The ring $ {\rm CH} (\st{H}_{g,1})$ is generated by two elements $t_1$, $t_2$ of degree~$1$.
		
		\item The pullback $ {\rm CH} (\st{H}_{g}) \longrightarrow  {\rm CH} (\st{H}_{g,1})$ is as follows.
		
		\begin{enumerate}
			
			\item If $g$ is even, it sends $c_{1}$ to $t_1+t_2$ and $c_{2}$ to $t_1t_2$ ($t_1$ and $t_2$ are the Chern roots of the vector bundle which defines the $c_i$'s).
			
			\item If $g$ is odd, it sends $\tau$ to $t_1$, $c_{2}$ to $-t_2^{2}$ and $c_{3}$ to $0$.
			
		\end{enumerate}
		
		\item The ideal of relations for $\st{H}_{g,1}$ is generated by the image of the ideal of relations for $\st{H}_g$.
		
	\end{enumerate}
\end{teorema}

The generators can be interpreted geometrically (see \hyperref[gen]{Section \ref{gen}}).

From this, with the results of Edidin--Fulghesu and Di Lorenzo we get the following.

\begin{corollario}(See \hyperref[p]{Theorem \ref{p}}, \hyperref[d]{Theorem \ref{d}})
	\begin{enumerate}
		
		\item If $g$ is even, then
		\[
		 {\rm CH} (\st{H}_{g,1}) = \frac{\fie{Z}[t_1,t_2]}{\bigl(2(2g+1)(t_1+t_2), g(g-1)(t_1^2+t_2^2) - 2g(g+3)t_1t_2\bigr)}\,.
		\]
		
		\item If $g$ is odd, then
		\[
		 {\rm CH}( \st{H}_{g,1} )= \frac{\fie{Z}[t_1,t_2]}{\bigl(4(2g+1)t_1,8t_1^{2} +2g(g+1)t_2^{2}\bigr)}\,.
		\]
		
	\end{enumerate}
\end{corollario}

To prove the main result we use equivariant techniques. The strategy is the following. Recall from \cite{ArVis}
 that $ \st{H}_{g}$ is a quotient $[X_{g}/G]$, where $X_{g} \subseteq \fie{A}(2g+2)$ is an open inside the space of binary form in two variables of degree $2g+2$, and $G$ is either ${\rm GL}_{2}$ (when $g$ is even), or $\fie{G}_m \times \fie{P}{\rm GL}_{2}$ (when $g$ is odd). In \hyperref[des]{Section \ref{des}} of the present paper we express $ \st{H}_{g,1}$ as a quotient $[Y_{g}/B]$, where $B \subseteq G$ is a Borel subgroup and $Y_{g}$ is an open subset of a $2g+3$-dimensional representation $\widetilde{\fie{A}}(2g+2)$ of $B$; the tautological map $ \st{H}_{g,1} \longrightarrow  \st{H}_{g}$ comes from a nonlinear finite flat $B$-equivariant map $\widetilde{\fie{A}}(2g+2) \longrightarrow  \fie{A}(2g+2)$. Thus, the Chow ring $ {\rm CH} (\st{H}_{g,1})$ is the equivariant Chow ring $ {\rm CH}_{B}(Y_{g})$; this easily proves parts (1) and (2) of the main theorem.

To prove part (3) one needs to projectivize (as is done in all the previous papers \cite{Vis3, EdFul, FulViv, DiLor});
 although, because the standard action of $\fie{G}_m$ does not commute with the action of $B$, one needs a weighted action, yielding a weighted projective stack, which maps to the projectivization $\fie{P}^{2g+2}$ of $\fie{A}(2g+2)$.  In \hyperref[red]{Section \ref{red}} we prove some technical results about Chow envelopes for quotient stacks  and we use them to compute the relations in \hyperref[eq]{Section \ref{eq}} and \hyperref[de]{Section \ref{de}}, which represent the technical heart of this paper. Finally, \hyperref[gen]{Section \ref{gen}} contains the geometric interpretation of the generators of the Chow group in both our cases.

\section*{Acknowledgements}
I wish to thank my advisor Angelo Vistoli for his patience in discussing with me and for always being available for both stupid and intelligent question.
I also wish to thank Andrea Di Lorenzo for helping me understand the geometric interpretation of the generators of the Chow ring.

\section{Description of $\st{H}_{g,1}$ as a quotient stack}\label{des}
Let us define precisely the actors of this paper. We will work over a fixed field $k$ of characteristic different from $2$. In all the paper the genus will be considered greater or equal than $2$.

\begin{definition}
	Let $S$ be a base scheme over $k$. A hyperelliptic curve of genus $g$ over $S$ is a morphism of $S$-schemes $C\rightarrow P\rightarrow S$ where $C\rightarrow S$ is a family of smooth curves of genus $g$, $P \rightarrow S$ is a Brauer-Severi scheme of relative dimension $1$ and $C\rightarrow P$ is finite and flat of degree $2$.
\end{definition}
	
 We define $\st{H}_g$ to be the fibered category in groupoid over the category of $k$-schemes whose objects are hyperelliptic curves of genus $g$. An arrow between $(C\rightarrow P \rightarrow S)$ and $(C'\rightarrow P'\rightarrow S')$ is a commutative diagram like the following: 
$$	\xymatrix{ C \ar[r] \ar[d] & P \ar[r] \ar[d] & S \ar[d] \\
	           C' \ar[r]  &      P' \ar[r]     &  S'  .} $$
	
	Let $\st{H}_{g,1}$ be the fibered category of hyperelliptic curves of genus g over $k$ with a marked point. An object $$(C \rightarrow P \rightarrow S, \sigma:S\rightarrow C)$$ in $\st{H}_{g,1}$ is a pair defined by $C\rightarrow P\rightarrow S$, a hyperelliptic curve of genus $g$ over $S$, and by $\sigma:S \rightarrow C$, a section of $C\rightarrow S$. A morphism between $(C\rightarrow P\rightarrow S,\sigma)$ and $(C'\rightarrow P'\rightarrow S',\sigma')$ is just an arrow in $\st{H}_g$ which commutes with the sections.

\begin{osservazione}
	Both $\st{H}_g$ and $\st{H}_{g,1}$ are Deligne-Mumford stacks: in fact the natural maps $\st{H}_{g} \rightarrow \st{M}_g$ and $\st{H}_{g,1} \rightarrow \st{M}_{g,1}$ are closed immersions, where $\st{M}_g$(respectively $\st{M}_{g,1}$), which is the stack of smooth curves of genus g (respectively of smooth curves with a marked point) is a Deligne-Mumford stack (see \cite[Theorem 8.4.5]{Oll}). Clearly the functor $\st{H}_{g,1} \rightarrow \st{H}_g$ forgetting the section is the universal curve over $\st{H}_g$.
\end{osservazione}
 As proved in \cite{ArVis}, the stack $\st{H}_g$ is equivalent to the fibered category  $\st{H}_g'$ defined as follows: an object $(P\rightarrow S, \st{L},i:\st{L}^{\otimes 2} \hookrightarrow \sh{O}_P)$ is defined by a Brauer-Severi scheme $P \rightarrow S$ of relative dimension $1$, an invertible sheaf $\sh{L}$ on $P$  and an injection $i$ such that $\sh{L}$ restricts to an invertible sheaf of degree $-(g+1)$ on any geometric fiber, the injection $i$ remains injective when restricted to any geometric fiber and the Cartier divisor $\Delta_i$ associated to the image of $i$ (called branch divisor) is smooth over $S$; an arrow between $(P\rightarrow S, \st{L},i:\st{L}^{\otimes 2} \hookrightarrow \sh{O}_P)$ and $(P'\rightarrow S', \st{L}',i':\st{L}'^{\otimes 2} \hookrightarrow \sh{O}_{P'})$ is a commutative diagram 
$$\xymatrix{ P \ar[r] \ar[d]^{\phi_0} & S \ar[d]\\
             P' \ar[r] &        S' }$$
 plus an isomorphism $\phi_1:\sh{L} \simeq \phi_0^*\sh{L}'$ of $\sh{O}_P$-modules such that the following diagram commutes:
 $$ \xymatrix{  \sh{L}^{\otimes 2} \ar[rr]^{\phi_1} \ar[dr]_{i} && \phi_0^*\sh{L}'^{\otimes 2} \ar[dl]^{\phi_0^*i'} \\ 
 	 & \sh{O}_P. & } $$
We can recover the morphism $C\rightarrow P$ as the morphism 
$$ \underline{\spec}_{\sh{O}_P}(\sh{O}_P \oplus \st{L}) \longrightarrow P$$
where $\sh{O}_P \oplus \st{L}$ is the $\sh{O}_P$-algebra defined by the injection $i$. In fact, given such injection, we can endow $\sh{O}_P \oplus \st{L}$ with a structure of $\sh{O}_P$-algebra where the multiplication is defined in the following way:
$$ (f,s) \cdot (f',s'):= (ff'+i(ss'),fs'+f's).$$
Consider an object $(C\rightarrow P \rightarrow S,\sigma) \in \st{H}_{g,1}(S)$. Using the description above, we only need to understand how to translate the information of the section $\sigma:S\rightarrow C$ in relation to the Brauer-Severi scheme $P\rightarrow S$, its invertible sheaf $\st{L}$ and the injection $i:\st{L}^{\otimes 2} \hookrightarrow \sh{O}_P$. 

First, we recall that given a morphism $\sigma_P:S\rightarrow P$ and an $\sh{O}_P$-algebra $\sh{A}$ one has the functorial bijective map 
		$$ {\rm Hom}_P(S,\underline{\spec}_{\sh{O}_P}(\sh{A})) \rightarrow {\rm Hom}_{\sh{O}_S-{\rm alg}}(\sigma_P^*(\sh{A}),\sh{O}_S). $$
Therefore, if we define $\sigma_P:=f \circ \sigma$ with $f:C\rightarrow P$ the finite flat morphism of degree $2$, the datum of the section $\sigma$ is equivalent to a pair $(\sigma_P,j)$ where $\sigma_P$ is a section of $P\rightarrow S$ and $$j \in {\rm Hom}_{\sh{O}_S-{\rm alg}}(\sh{O}_S\oplus\sigma_P^*(\st{L}),\sh{O}_S).$$
We notice that ${\rm Hom}_{\sh{O}_S-{\rm alg}}(\sh{O}_S\oplus\sigma_P^*(\st{L}),\sh{O}_S)$ is the subset of ${\rm Hom}_{\sh{O}_S-{\rm mod}}(\sigma_P^*(\st{L}),\sh{O}_S)$ such that 
the two maps $\sigma_P^*(i):\sigma_P^*(\st{L}^{\otimes 2}) \rightarrow \sh{O}_S$ and 
$j^{\otimes 2}:\sigma_P^*(\st{L})^{\otimes 2} \rightarrow \sh{O}_S$ coincide (up to the canonical isomorphism $\sigma_P^*(\st{L}^{\otimes 2}) \cong \sigma_P^*(\st{L})^{\otimes 2}$).

Thus we define $\st{H}_{g,1}'$ as the category fibred in groupoids  whose objects are $$(P\rightarrow S, \st{L},i:\st{L}^{\otimes 2} \hookrightarrow \sh{O}_P, \sigma_P, j)$$
where $(P\rightarrow S, \st{L},i) \in \st{H}_g'(S)$, $\sigma_P$ is a section of $P\rightarrow S$ and $j: \sigma_P^*(\st{L}) \rightarrow \sh{O}_S$ is a morphism of $\sh{O}_S$-modules such that $j^{\otimes 2}= \sigma_P^*(i)$.
The morphisms are defined in the natural way.

We have just proved the following statement.
\begin{proposizione}
	There is an equivalence of fibred categories in groupoids between $\st{H}_{g,1}$ and $\st{H}_{g,1}'$.
\end{proposizione}

For the sake of simplicity, the section of the Brauer-Severi scheme will be denoted just by $\sigma$.  We denote by $\sigma_{\infty}$ the section $S\rightarrow\fie{P}_S^1$ defined by the map $S\rightarrow \spec k \hookrightarrow \fie{P}^1_k $ sending $S$ to $[0:1]$ in $\fie{P}_k^1$. The next step will be to describe this stack $\st{H}_{g,1}'$ as a quotient stack. Let ${\rm H}_{g,1}'$ be the auxiliary fibred category whose objects over a base scheme $S$ are given as pairs consisting of an object $(P\rightarrow S, \st{L},i, \sigma, j)$ in $\st{H}_{g,1}'(S)$, plus an isomorphism $$\phi:(P,\sh{L},\sigma) \simeq (\fie{P}_S^1,\sh{O}(-g-1),\sigma_{\infty})$$ 
which consists of an isomorphism of $S$-schemes $\phi_0:P \simeq \fie{P}_S^1$ with the property that $\phi_0 \circ \sigma = \sigma_{\infty}$, plus an isomorphism $\phi_1:\st{L}\simeq \phi_0^*\sh{O}(-g-1)$. The arrows in ${\rm H}_{g,1}'$ are arrows in $\st{H}_{g,1}'$ preserving the isomorphism $\phi$.
\begin{osservazione}\label{re}
Clearly, ${\rm H}_{g,1}'$ is a category fibred in groupoids over the category of $k$-schemes and it is straightforward to verify that the groupoid ${\rm H}_{g,1}'(S)$ is in fact equivalent to a set for every $k$-scheme $S$. This implies that ${\rm H}_{g,1}'$ is equivalent to a functor.
Notice that we have an action of the group scheme  $\underline{{\rm Aut}}_k(\fie{P}_k^1,\sh{O}(-g-1),\sigma_{\infty})$ on the functor ${\rm H}_{g,1}'$ defined  by composing $\phi:(P,\sh{L},\sigma) \simeq (\fie{P}_S^1,\sh{O}(-g-1),\sigma_{\infty})$ with an element of the group ${\rm Aut}_S(\fie{P}_S^1,\sh{O}(-g-1),\sigma_{\infty})$ for every $S$-point.
\end{osservazione}
Before giving the description of $\st{H}_{g,1}$ as a quotient stack, let us introduce some notation.
Let $\fie{A}(n)$ be the affine space of homogenous polynomials in two variables of degree $n$, which is an affine space of dimension $n+1$, and let  $\fie{A}_{sm}(n)$  be the open affine subscheme of $\fie{A}(n)$ defined as the complement of the discriminant locus. 

We consider the closed subscheme $\widetilde{\fie{A}}(n) \hookrightarrow \fie{A}(n) \times\fie{A}^1$ defined as
$$ \widetilde{\fie{A}}(n)(S) := \{ (f,s) \in (\fie{A}(n)\times\fie{A}^1)(S) \vert f(0,1)=s^2 \}. $$ We denote by $\widetilde{\fie{A}}_{sm}(n)$ the intersection of $\fie{A}_{sm}(n)\times \fie{A}^1$ with $\widetilde{\fie{A}}(n)$ inside $\fie{A}(n)\times \fie{A}^1$, seeing it as an open subscheme of $\widetilde{\fie{A}}(n)$.

\begin{proposizione}
	We prove the following:
\begin{enumerate}

	\item the group scheme $\underline{{\rm Aut}}(\fie{P}^1,\sh{O}(-g-1),\sigma_{\infty})$ is isomorphic to ${\rm B}_{2}/\mu_{g+1}$ where ${\rm B}_{2}$ is the subgroup of lower triangular matrices inside ${\rm GL}_{2}$ and $\mu_{g+1}\hookrightarrow {\rm B}_2$ is the natural inclusion inside the subgroup of the diagonal matrices of $\mu_{g+1}$, the group of $(g+1)$-th roots of unity.
	
	\item The functor ${\rm H}_{g,1}'$ is naturally isomorphic to $\widetilde{\fie{A}}_{sm}(2g+2)$.
	
	\item The action of $\underline{{\rm Aut}}(\fie{P}^1,\sh{O}(-g-1),\sigma_{\infty})$ on ${\rm H}_{g,1}'$ translates into the action of ${\rm B}_2/\mu_{g+1}$ on $\widetilde{\fie{A}}_{sm}(2g+2)$ defined by 
	
	$$A\cdot\Big(f(\underline{x}),s\Big):=\Big(f\big(A^{-1}\underline{x}\big),c^{-(g+1)}s\Big)$$ where 
$$A=	\begin{bmatrix}
		a & 0 \\
		b & c
	\end{bmatrix} \in  {\rm B}_2/\mu_{g+1}.$$
\end{enumerate}
\end{proposizione}

\begin{proof}
	  There is a natural isomorphism (c.f. proof of \cite[Theorem 4.1]{ArVis})
		$$\underline{{\rm Aut}}(\fie{P}^1,\sh{O}(-g-1)) \longrightarrow {\rm GL}_2/\mu_{g+1} $$
		which follows from the exact sequence of sheaves of groups 
 $$ \xymatrix { 0 \ar[r] & \mu_{g+1} \ar[r] & \underline{{\rm Aut}}(\fie{P}^1,\sh{O}(1)) \ar[r]^<<<<<{\alpha} & \underline{{\rm Aut}}(\fie{P}^1,\sh{O}(-g-1)) \ar[r] & 0} $$ 
 where $\alpha(\phi_0,\phi_1)= (\phi_0,\phi_1^{\otimes(-g-1)}) $. The same exact sequence leads us to the isomorphism 
 $$\underline{{\rm Aut}}(\fie{P}^1,\sh{O}(-g-1),\sigma_{\infty})\simeq \underline{{\rm Aut}}(\fie{P}^1,\sh{O}(1),\sigma_{\infty})/\mu_{g+1}.$$
 If we identify $\underline{{\rm Aut}}(\fie{P}^1,\sh{O}(1))$ with ${\rm GL}_2$, the subgroup $\underline{{\rm Aut}}(\fie{P}^1,\sh{O}(1),\sigma_{\infty})$ corresponds to ${\rm B}_2$ inside ${\rm GL}_2$. This proves the first claim.
 
 Consider an element $(P\rightarrow S, \st{L},i, \sigma, j,\phi)$ in ${\rm H}_{g,1}'(S)$: the pushforward of the inclusion $i:\st{L}^{\otimes 2}\hookrightarrow \sh{O}_P$ along $\phi=(\phi_0,\phi_1)$induces an inclusion
 $$  \sh{O}_{\fie{P}_S^1}(-2g-2) \hookrightarrow \sh{O}_{\fie{P}_S^1} $$ 
 which we keep denoting $i$. We identify such inclusion with an element $f  \in {\rm H}^0(\fie{P}_S^1,\sh{O}_{\fie{P}_S^1}(2g+2))=\fie{A}(2g+2)(S)$. Observe that actually $f$ belongs to $\fie{A}_{sm}(2g+2)$ because by construction the divisor associated to $f$ has to be smooth over $S$. Using again the isomorphism $\phi$ we can describe $j$ as an element of ${\rm Hom}_{\sh{O}_S}(\sigma_{\infty}^*\sh{O}_{\fie{P}_S^1}(-g-1),\sh{O}_S)$ such that $j^{\otimes 2}= \sigma_P^*(i)$, or equivalently as an element $s \in H^0(S,\sigma_{\infty}^*\sh{O}_{\fie{P}_S^1}(g+1))$ such that $\sigma_{\infty}^*(f)=s^{\otimes 2}$.
 We have a non-canonical isomorphism $\sigma_{\infty}^*\sh{O}_{\fie{P}_S^1}(g+1)\simeq \sh{O}_S$ given by the association $f \mapsto f(0,1)$, therefore we are considering $j$ as an element $s$ in $H^0(S,\sh{O}_S)=\fie{A}^1(S)$. In the same way, given $f \in \fie{A}_{sm}(2g+2)$ induced by the inclusion $i$, we have that $\sigma_{\infty}^*(i)$ will be represented by $f(0,1)$ in $\fie{A}^1(S)$. The condition above is represented through this identification by $f(0,1)=s^2$. This gives us a base-preserving functor from ${\rm H}_{g,1}'$ to $\widetilde{\fie{A}}_{sm}(2g+2)$, seeing it as a closed subscheme of $\fie{A}_{sm}(2g+2) \times \fie{A}^1$. There is also a base-preserving functor in the other direction sending an element $(f,s) \in \widetilde{\fie{A}}_{sm}(2g+2)(S)$ to the object inside ${\rm H}_{g,1}'(S)$ of the form 
 $$ \big( \fie{P}_S^1 \rightarrow S,\sh{O}(-g-1),f,\sigma_{\infty},s, {\rm id} \big). $$ 
It is straightforward to see that it is a quasi-inverse to the previous functor. This proves the second claim.
 
As far as the action is concerned, it is a classical fact that in general, given an automorphism of $(\fie{P}^1,\sh{O}(1))$ expressed by a matrix $A$ in ${\rm GL}_2$,
 the corresponding automorphism of $(\fie{P}^1,\sh{O}(-1))$ is expressed by the matrix $A^{-1}$ and, tensoring $n$ times, we get an automorphism of $(\fie{P}^1,\sh{O}(-n))$ defined by 
 $$f(x) \longmapsto f(A^{-1}x).$$ 
 Since the equation $f(0,1)=s^2$ is invariant for the action of the group ${\rm GL}_2$ on $\fie{A}(2g+2)\times \fie{A}^1$ described above, we get an induced action on $\widetilde{\fie{A}}(2g+2)$ and this clearly proves the third claim, since we are restricting the ${\rm GL}_2$-action to the Borel subgroup.

\end{proof}

\begin{proposizione}
	If we denote the group scheme $\underline{{\rm Aut}}_k(\fie{P}_k^1,\sh{O}(-g-1),\sigma_{\infty})\,$ by $G$, then the natural forgetting map 
	$$ {\rm H}_{g,1}' \longrightarrow \st{H}_{g,1}'$$ 
	is a $G$-torsor and in particular $\st{H}_{g,1}' \simeq \big[ {\rm H}_{g,1}' /G \big]$.
\end{proposizione}
\begin{proof}
	Consider an object $(P\rightarrow S, \st{L},i, \sigma, j)$ in $\st{H}_{g,1}'(S)$, then we can find an fppf covering $S'\rightarrow S$ such that there exists an isomorphism $\phi$ between the pullback to $S'$ of the pair $(P,\st{L})$  and $(\fie{P}_{S'}^1,\sh{O}(-g-1))$.
    We denote by $\tilde{\sigma}$ the composition $\phi_0 \circ \sigma$. Using the transitivity of the action of ${\rm GL}_2$ over $\fie{P}^1$, we can find an element $T$ of ${\rm GL}_2(S)$, up to passing to a fppf covering again, which sends $\tilde{\sigma}$ to $\sigma_{\infty}$. This implies that, given an atlas $H$ of $\st{H}_{g,1}'$ we can find a fppf covering of $H$ such that the pullback of the morphism $${\rm H}_{g,1}' \longrightarrow \st{H}_{g,1}'$$ through this covering is a trivial $G$-torsor. Therefore this concludes the proof.
\end{proof}

We denote by $\rm{PB}_2$ the Borel subgroup of ${\rm PGL}_2$.
Using \cite[Proposition 4.4]{ArVis}, we get the following proposition. 

\begin{proposizione}\label{a}
	 Let $g \geq 2$ be an integer.
	\begin{itemize}
	 \item[i)]If $g$ is even, then the homomorphism of group schemes
	 $$ {\rm B}_2/\mu_{g+1} \longrightarrow {\rm B}_2 $$ 
	 defined by  $$[A] \mapsto {\rm det}(A)^{g/2}A$$ is an isomorphism.
	 
	 \item[ii)]If $g$ is odd, then the homomorphism of group schemes
	 $$ {\rm B}_2/\mu_{g+1} \longrightarrow \fie{G}_m \times\rm{PB}_2 $$
	 defined by
	 $$ [A] \mapsto ({\rm det}(A)^{(g+1)/2},[A])$$
	 is an isomorphism.
	\end{itemize}
\end{proposizione}

Putting together all the precedent results, we finally get the description we need.

\begin{corollario} Let $g \geq 2$ an integer. The stack $\st{H}_{g,1}$ is equivalent to the quotient stack 
	$$ \Big[ \widetilde{\fie{A}}_{sm}(2g+2)/ G\Big] $$ 
	 where the group $G$ and its action on $\widetilde{\fie{A}}_{sm}(2g+2)$ are described by the following formulas:
    \begin{itemize}
    	\item if $g$ is even, then $G = {\rm B}_2$ and the action is given by
    	            $$ A\cdot(f(x),s):= \bigg(({\rm det}A)^gf(A^{-1}x),a^{\frac{g}{2}}{c^{-\frac{g+2}{2}}}s\bigg) $$ 
    	             where 
    	            $$A=	\begin{pmatrix}
    	            a & 0 \\
    	            b & c
    	            \end{pmatrix} \in  {\rm B}_2.$$
    	\item if $g$ is odd, then $G = \fie{G}_m \times\rm{PB}_2$ and the action is given by
    	$$ (\alpha,A)\cdot(f(x),s):= \Big(\alpha^{-2}{\rm det}(A)^{g+1}f(A^{-1}x),\alpha^{-1}a^{\frac{g+1}{2}}{c^{-\frac{g+1}{2}}}s\Big)$$ 
    	where  $$A=	\begin{bmatrix}
    	a & 0 \\
    	b & c
    	\end{bmatrix} \in  \rm{PB}_2.$$
    \end{itemize}

	\end{corollario}

\begin{osservazione}
	Notice that in both the even and odd genus case, the group $G$ is in fact a Borel (maximal connected solvable) subgroup of $\rm{GL}_2/\mu_{g+1}$.
\end{osservazione}	

 \section{Reduction to the weighted projectivization }\label{red}

From now on $G$ will be one of the two groups described in \hyperref[a]{Proposition \ref{a}} depending on the parity of the genus.

We have now found the description of $\st{H}_{g,1}'$ as a quotient stack. Using this presentation, we know that the integral Chow ring of $\st{H}_{g,1}'$ can be computed as the $G$-equivariant Chow ring of $\widetilde{\fie{A}}_{sm}(2g+2)$ as defined in \cite[Proposition 19]{EdGra}, i.e. 
$$ {\rm CH}^*(\st{H}_{g,1}')= {\rm CH}^*_{G}(\widetilde{\fie{A}}_{sm}(2g+2)).$$

The following remark explains why we can reduce ourselves to the computation of $T$-equivariant Chow ring of $\widetilde{\fie{A}}_{sm}(2g+2)$, where $T$ is the maximal torus of the diagonal matrices inside $G$.
\begin{osservazione}
	Recall that the group $G$ is defined as:
	\begin{itemize}
		\item $G = \fie{G}_m \times\rm{PB}_2$ if $g$ is odd,
		\item $G = {\rm B}_2$ if $g$ is even.
	\end{itemize}
	We notice that both of them are unipotent split extensions of a 2-dimensional split torus $T$ because $G$ is a Borel sungroup of $\rm{GL}_2/\mu_{g+1}$. If $g$ is odd, we construct the following isomorphism explicitly for the sake of explicit computations:
	$$  \fie{G}_m \times\rm{PB}_2 \longrightarrow \fie{G}_m^2 \ltimes \fie{G}_a $$
	defined by 
	$$ (\alpha, [A]) \mapsto (\alpha, a/c, b/c) $$ 
	where 
	$$A=	\begin{bmatrix}
	a & 0 \\
	b & c
	\end{bmatrix} \in  \rm{PB}_2.$$

	Using the result \cite[Lemma 2.3]{RoVis}, we deduce that the homomorphism 
	$$ {\rm CH}^*_T(X) \longrightarrow {\rm CH}^*_G(X) $$
	induced by the projection $G \rightarrow T$ is in fact an isomorphism of rings for every smooth $G$-scheme $X$ (it is in fact an isomorphism of graded groups for $X$ an algebraic scheme). Therefore we can consider directly the action of the $2$-dimensional split torus $T$ inside $G$.

	Using this identification, we get the following description of the action of $T$ on the affine scheme $\fie{A}(2g+2) \times \fie{A}^1$:
	\begin{itemize}
		\item if $g$ is even, 
		$$ (t_0,t_1)\cdot(f(x_0,x_1),s):= \Big((t_0t_1)^{g} f(x_0/t_0,x_1/t_1), t_0^{\frac{g}{2}}t_1^{-\frac{g+2}{2}}s \Big) ; $$
		\item if $g$ is odd, 
		$$ (\alpha,\rho)\cdot(f(x_0,x_1),s) = \Big(\alpha^{-2}\rho^{g+1}f(x_0/\rho,x_1),\alpha^{-1}\rho^{\frac{g+1}{2}}s\Big).$$ 
	\end{itemize}
\end{osservazione}

From now on, we have to concentrate on computing the $T$-equivariant Chow ring of $\widetilde{\fie{A}}_{sm}(2g+2)$, where $T$ will be the split 2-dimensional torus and the action will be the one described above depending on whether $g$ is odd or even. 

We will use the localization sequence to compute the Chow group we are interested in. In fact, if we manage to describe $\widetilde{\fie{A}}_{sm}(2g+2)$ as an open subscheme $U$ of a $T$-representation $V$, the localization sequence will give us the following explicit description:
$$ {\rm CH}_T(\widetilde{\fie{A}}_{sm}(2g+2)) = \frac{{\rm CH}({\rm B}T)}{I} $$
where $I$ will be the ideal generated by the pushforward of the cycles from the closed subscheme $V\setminus U$.

\begin{osservazione}

There is a natural isomorphism of $k$-schemes (without considering the $T$-action) $$\xi_n:\widetilde{\fie{A}}(n) \simeq \fie{A}^{n+1}$$ described by the formula 
$$ \xi_n(a_0,\dots,a_n,s)=(a_0,\dots,a_{n-1},s).$$
 Therefore $\widetilde{\fie{A}}(2g+2)$ can be identified  with the $k$-scheme $\fie{A}^{2g+3}$ endowed with the unique action that makes $\xi_{2g+2}$ into a $T$-equivariant isomorphism. 
 Under this identification, $\widetilde{\fie{A}}(2g+2)$ is clearly a $T$-representation and the natural projection map $$\varphi_{2g+2} :\widetilde{\fie{A}}(2g+2) \hookrightarrow \fie{A}(2g+2) \times \fie{A}^1 \longrightarrow \fie{A}(2g+2) $$is $T$-equivariant. Recall that $\widetilde{\fie{A}}_{sm}(n)$ is defined as the intersection of $\fie{A}_{sm}(n) \times \fie{A}^1$ and $\widetilde{\fie{A}}(n)$ inside $\fie{A}(n)\times \fie{A}^1$; thus, 
 we have that $$\widetilde{\fie{A}}_{sm}(2g+2)=\varphi_{2g+2}^{-1}(\fie{A}_{sm}(2g+2)).$$
Therefore $\widetilde{\fie{A}}_{sm}(2g+2)$ is a $T$-invariant open subset of the $T$-representation $\widetilde{\fie{A}}(2g+2)$. If we denote by $\widetilde{\Delta}$ the complement of $\widetilde{\fie{A}}_{sm}(2g+2)$ inside $\widetilde{\fie{A}}(2g+2)$, then set-theoretically the following holds:
$$ \widetilde{\Delta}= \varphi_{2g+2}^{-1}(\Delta) $$
where $\Delta$ is the discriminant locus inside $\fie{A}(2g+2)$.
\end{osservazione}

The problem now is to describe the image of the pushforward along the inclusion
 $$\widetilde{\Delta} \hookrightarrow \widetilde{\fie{A}}(2g+2)$$ 
 at the level of Chow group. The idea is to construct a stratification of $\widetilde{\Delta}$ lifting the one introduced in \cite[Proposition 4.1]{EdFul}. We will pass to the projectivization of $\widetilde{\fie{A}}(2g+2)$, considering $\widetilde{\fie{A}}(2g+2)$ 
  as $T$-representation through the identification $\xi_{2g+2}$. Notice that $\widetilde{\Delta}$ is invariant for the action of $\fie{G}_m$ defined by the formula $\lambda\cdot (h,t):=(\lambda^2 h,\lambda t)$. We will denote by $\fie{P}(2^N,1)$ the quotient stack of $\widetilde{\fie{A}}(N)\setminus 0$ by $\fie{G}_m$ using this weighted action, where $N:=2g+2$.
 
 The following proposition explains why we can pass to the weighted projective setting without losing any information about Chow group. First, suppose we have two group schemes $G$ and $H$ acting on a scheme $X$ such that their actions commute. From now on, the Chow group ${\rm CH}([X/(G\times H)])$ defined as the  $G \times H$-equivariant Chow group of $X$ will be also denoted by ${\rm CH}_H([X/G])$ or ${\rm CH}_G([X/H])$.
\begin{proposizione}\label{GM}
	Let $X$ be a smooth algebraic scheme over $k$ with an action of $\fie{G}_m$ and an action of a group $G$ such that the two actions commute, then the natural morphism of rings
	$$ {\rm CH}([X/(\fie{G}_m\times G)]) \longrightarrow  {\rm CH}([X/G])$$
	induced by the pullback along the $\fie{G}_m$-torsor 
	$$ [X/G] \longrightarrow [X/(\fie{G}_m\times G)]$$
	is surjective. Moreover, its kernel is the ideal generated by $c_1(\st{L})$ in ${\rm CH}([X/(\fie{G}_m\times G)])$, where $\st{L}$ is the line bundle associated to the $\fie{G}_m$-torsor.
\end{proposizione}

\begin{proof}
	 Because torsors are stable under base change and representable as morphisms of stacks, we can reduce to the case of an $\fie{G}_m \times G$-equivariant $\fie{G}_m$-torsor in the category of algebraic spaces, where this result is well known.
\end{proof}
 
 Using the previous proposition and writing down the following commutative diagram of ${\rm CH}({\rm B}T)$-algebras:
 $$\xymatrix{ {\rm CH}_T(\widetilde{\Delta} \setminus 0) \ar[r] & {\rm CH}_T(\widetilde{\fie{A}}(N)\setminus 0) \ar[r] & {\rm CH}(\st{H}_{g,1})) \ar[r] & 0 \\ 
 	{\rm CH}_T([\widetilde{\Delta}\setminus 0/\fie{G}_m]) \ar[r] \ar[u] & {\rm CH}_T(\fie{P}(2^N,1)) \ar[r] \ar[u] & {\rm CH}_T([\widetilde{\fie{A}}_{sm}(N)\setminus 0/\fie{G}_m]) \ar[r] \ar[u] & 0 \ } $$
 we can reduce the computation to the weighted projective setting and then set the first Chern class of the line bundle associated to the $\fie{G}_m$-torsor equal to $0$. 
 The rest of the section will be dedicated to computing the $T$-equivariant Chow group of $\fie{P}(2^N,1)$, i.e. 
  $$ {\rm CH}_{T}(\fie{P}(2^N,1)):={\rm CH}_{T \times \fie{G}_m}(\widetilde{\fie{A}}(N)\setminus 0). $$
Let $T$ be a split torus of dimension $r$, i.e. $T\simeq \fie{G}_m^r$. 
 \begin{osservazione}
 	 Edidin and Graham have already proved in \cite[Section 3.2]{EdGra} that 
 	$${\rm CH}({\rm B}T) \simeq  \fie{Z}[T_1,\dots,T_r] $$ 
 	where $T_j=c_1^{(\fie{G}_{m})_j}(\fie{A}^1)$, $(\fie{G}_m)_j$ is the $j$-th factor of the product $T$ and $\fie{A}^1$ is the representation of $\fie{G}_m$ with weight $1$. 
 	Suppose $T$ acts on $\fie{A}^{n+1}$. We can decompose the representation $\fie{A}^{n+1}$ in a product of irreducible representations $\fie{A}^1_0 \times \dots\times \fie{A}^1_n$ where $T$ acts on $\fie{A}^1_i$ with some weights $(m_1^i,\dots,m_r^i) \in \fie{Z}^r$ for every $0\leq i\leq n$. If we denote by $ p_i(T_1,\dots,T_r)$ the first Chern classes $c_1^T(\fie{A}^1_i)$, then we get 
 	$$ p_i(T_1,\dots,T_r) = \sum_{j=1}^r m_j^i T_j \in \fie{Z}[T_1,\dots,T_r]$$ 
 	for every $0\leq i \leq n$.
 \end{osservazione}

\begin{proposizione}\label{c}
	In the setting of the previous remark, we get the following result:
	$$ {\rm CH}_T(\fie{A}^{n+1}\setminus 0) = \frac{\fie{Z}[T_1,\dots,T_r]}{\Big(\prod_{i=0}^{n}p_i(T_1,...,T_r)\Big)}. $$
\end{proposizione}

\begin{proof}
Consider the localization sequence for the $T$-invariant open subscheme $\fie{A}^{n+1} \setminus 0 \hookrightarrow \fie{A}^{n+1}$:

$$ \xymatrix{{\rm CH}({\rm B}T) \ar[r]^{(i_0)_*} & {\rm CH}_T(\fie{A}^{n+1}) \ar[d]^{(i_0)^*} \ar[r] & {\rm CH}_T(\fie{A}^{n+1}\setminus 0) \ar[r] & 0\\ 
 & {\rm CH}({\rm B}T)  } $$ 

where $i_0$ is the $0$-section of the $T$-equivariant vector bundle over $\spec k$. Using the self-intersection formula, we get that the image through $(i_0)_*$ of ${\rm CH}({\rm B}T)$ is the ideal (inside ${\rm CH}({\rm B}T)$) generated by $c_{n+1}^T(\fie{A}^{n+1})$. Thus we get the following equality 
$$c_{n+1}^T(\fie{A}^{n+1}) = \prod_{i=0}^n c_1^T(\fie{A}_i^1) = \prod_{i=0}^{n}p_i(T_1,\dots,T_r)$$
and the statement follows.
\end{proof}

Recall that 
$$ \fie{P}(2^N,1) \simeq  \Big[ \big(\widetilde{\fie{A}}(N)\setminus 0\big) / \fie{G}_m \Big] $$ 
where $\fie{G}_m$ acts with weights $(2,\dots,2,1)$ and we have an action of a $2$-dimensional split torus $T$ over $\widetilde{\fie{A}}(N)$. Furthermore, recall that by definition
$$ {\rm CH}_T(\fie{P}(2^N,1)):= {\rm CH}_{T \times \fie{G}_m}(\widetilde{\fie{A}}(N)\setminus 0). $$
 We denote by $p_i(T_0,T_1)$ the first $T$-equivariant Chern class of the $i$-th factor of the product $\widetilde{\fie{A}}(N) \simeq \fie{A}^1_0 \times \dots\times \fie{A}^1_N$, and by $h$ the first $\fie{G}_m$-equivariant Chern class of the irreducible representation of $\fie{G}_m$ with weight $1$ (here we are considering $\widetilde{\fie{A}}(N)$ a $T$-representation through the identification $\xi_N$). The previous proposition gives us the following result.

\begin{corollario}
	In the above setting, we get 
	$$ {\rm CH}_T(\fie{P}(2^N,1)) \simeq \frac{\fie{Z}[T_0,T_1,h]}{(h+p_N(T_0,T_1))\prod_{i=0}^{N-1}(2h+p_i(T_0,T_1))} $$ 
	for some homogeneous polynomials $p_i(T_0,T_1)$ of degree $1$. 
\end{corollario}

\begin{osservazione}
	We can easily compute the polynomials $p_i$ mentioned in the previous corollary, but it is not necessary.
\end{osservazione}

\section{Equivariant Chow envelope for $\widetilde{\Delta}$}\label{eq}

In this section, we recall briefly the theory of Chow envelopes for quotient stacks and then we find a Chow envelope for $\widetilde{\Delta}$. The idea is to modify the one described in \cite[Section 4]{EdFul} so to suit this weighted projective setting. Fix again $N:=2g+2$. 
\begin{definition}
	Let $f:\st{X} \rightarrow \st{Y}$ be a proper, representable morphism of quotient stacks. We say that $f$ has the property $\sh{N}$ if the morphism of groups 
	$$ f_*:{\rm CH}(\st{X}) \rightarrow {\rm CH}(\st{Y}) $$
	is surjective.
	
	We say that a morphism of algebraic stacks $f:\st{X} \rightarrow \st{Y}$ is a Chow envelope if $f(K): \st{X}(K) \rightarrow \st{Y}(K)$ is essentially surjective for every extension of fields $K/k$, i.e. for every element $y \in \st{Y}(K)$ there exist an object $x \in \st{X}(K)$ and an isomorphism $\eta: f(K)(x) \rightarrow y$ in the groupoid $Y(K)$.
\end{definition}
\begin{osservazione}\label{Env}
	It is a classical fact that if $f:X \rightarrow Y$ is a proper morphism of algebraic spaces over $k$ such that $f$ is a Chow envelope, then $f$ has the property $\sh{N}$.
\end{osservazione}
We want to prove the same for quotient stacks.
\begin{lemma}\label{lem2}
	Let $G$ be a group scheme over $k$ and suppose we have an action of $G$ on two algebraic spaces $X$ and $Y$ and a map $f: X \rightarrow Y$  which is $G$-equivariant. We denote by $f^G$ the induced morphism of quotient stacks and we assume it is proper and representable. If $f^G$ is a Chow envelope, then $f^G$  has the property $\sh{N}$. 
\end{lemma}

\begin{proof}
	We need to prove that 
	$$f_*^G: {\rm CH}_i^G(X) \longrightarrow {\rm CH}_i^G(Y) $$ 
	is surjective for every $i \in \fie{N}$. Fix $i \in \fie{N}$. We consider an approximation $U \subset V$ where ${\rm codim}_V(V\setminus U)>i$ and $G$ acts freely on $U$, therefore $X \times U/G$ is an algebraic space and  
	$$ {\rm CH}_i^G(X)={\rm CH}_{i+l-g}(X \times U/G). $$ 
	 If we consider the following cartesian diagram 
	$$ \xymatrix{ (X \times U)/G \ar[r]^{f_U} \ar[d] & (Y \times U)/G \ar[d]  \\    [X/G] \ar[r]^{f^G}  & [Y/G]  } $$ 
	we get that $f_U(K)$ is surjective for every  extension of fields $K/k$ because $f^G$ has the same property and being surjective is stable under base change. \hyperref[Env]{Remark 3.2} implies the surjectivity of $(f_U)_*$ and therefore of $f_*^G$.
\end{proof}

\begin{osservazione}
	The previous lemma is a natural corollary of \cite[Lemma 3.3]{EdGra}, using the fact that being a Chow envelope is a property invariant under base change.
\end{osservazione}

We recall that a special group scheme $T$ is a group scheme (over $k$) such that every $T$-torsor $P\rightarrow S$ is locally trivial in the Zariski topology, i.e. there exists a Zariski covering $\{U_i \rightarrow S\}_{i\in I}$ such that $P\times_{S} U_i \rightarrow U_i$ is a trivial $T$-torsor for every $i \in I$.
\begin{osservazione}\label{boh}
	Given a special group $T$ acting on an algebraic space $X$ over $k$, then the $T$-torsor $$X \rightarrow [X/T]$$ is clearly a Chow envelope thanks to $T$ being special.
\end{osservazione}
\begin{corollario}\label{ChEnv}
	Let $G,T$ be two group schemes over $k$ and suppose we have an action of $G \times T$ on two algebraic spaces $X$ and $Y$ and a map $f: X \rightarrow Y$  which is $G\times T$-equivariant. Assume that $T$ is a special group. Suppose that $f^G$ is a Chow envelope and $f^{G\times T}$ is a proper representable morphism. Then $f^{G \times T}: [X/(G\times T)] \rightarrow [Y/(G \times T)]$  has the property $\sh{N}$.
\end{corollario}

\begin{proof}
	If we consider the cartesian diagram of quotient stacks 
	$$ \xymatrix{ [X/G] \ar[r]^{f^G} \ar[d] & [Y/G] \ar[d] \\ 
	              [X/(G\times T)] \ar[r]^{f^{G\times T}} & [Y/(G \times T)]} $$
    we notice that the two vertical maps are $T$-torsors. The \hyperref[boh]{Remark 3.5} easily implies that $f^{G\times T}$ is a Chow envelope. Therefore $f^{G\times T}$ has the property $\sh{N}$.
\end{proof}

Let us recall the setting we are studying. We have a morphism  
$$ \varphi_n:\widetilde{\fie{A}}(n) \rightarrow \fie{A}(n) $$ 
which is defined on points as $(f,s) \mapsto f$; it induces a morphism $\phi_n:\fie{P}(2^n,1) \rightarrow \fie{P}^n$ for every $n \in \fie{N}$. In the case of smooth hyperelliptic curves (without the datum of the section), we have the following exact sequence:
$$ \xymatrix{{\rm CH}_{{\rm GL}_2}(\Delta) \ar[r] & {\rm CH}_{{\rm GL}_2}(\fie{A}(2g+2)) \ar[r] & {\rm CH}(\st{H}_g) \ar[r] & 0 } $$ 
where $\Delta$ is the discriminant locus, which is naturally $\fie{G}_m$-invariant with respect to the standard action (see \cite[Corollary 4.7]{ArVis}). By abuse of notation we denote by $\Delta$ the image of the discrimant locus in $\fie{P}^N$. We know that, if we define 
$$ \Delta_r:=\{ h \in \fie{P}^N \vert h=f^2g \textrm{  in some field extension, with deg}(f)=r \},$$ 
 for every $r\leq N/2$ (recall that $N:=2g+2$), the chain of closed subsets $\Delta_{r+1} \subset \Delta_r$ is a stratification of $\Delta$ and the coproduct of the maps
$$ \pi_r: \fie{P}^r \times \fie{P}^{N-2r} \longrightarrow \fie{P}^N$$ 
defined by $\pi_r(h,g)=h^2g$ forms a Chow envelope for $\Delta$ when ${\rm char}(k)>N-2$. (see \cite[Lemma 3.2]{Vis3}). 

If we denote the projectivization of $\widetilde{\Delta}$ in $\fie{P}(2^N,1)$ by $\overline{\Delta}$, we clearly have $\overline{\Delta}=\phi_N^{-1}(\Delta)$ (set-theoretically). Thus, we consider the pullback of the stratification of $\Delta$ through the map $\phi_N$. To be precise, we have the stratification $\overline{\Delta}_{r+1}\subset \overline{\Delta}_r $ given by the following:
$$ \overline{\Delta}_r:=\{(h,s) \in \fie{P}(2^N,1) \textrm{ such that, in some field extension, } h=f^2g \textrm{ with } {\rm deg}(f) = r\}.$$

Our aim is to construct a Chow envelope (of quotient stacks) for $\overline{\Delta} \subset \fie{P}(2^N,1)$. We recall that $\widetilde{\fie{A}}(n)$ is defined as the closed subscheme of $\fie{A}(n) \times \fie{A}^1$  given by the pairs $$(f,s) \in \fie{A}(n) \times \fie{A}^1$$ satisfying the equation $f(0,1)=s^2$. 
We consider the morphisms of schemes 
$$ c_r:\fie{A}(r) \times \widetilde{\fie{A}}(N-2r) \longrightarrow \widetilde{\fie{A}}(N) $$ 
defined by
 \begin{equation*}\label{c_r}
    (f,(g,s)) \mapsto (f^2g,f(0,1)s). \tag{*}
 \end{equation*} 
 Furthermore, we endow $\fie{A}(r) \times \widetilde{\fie{A}}(N-2r)$ with a $T$-action.

  \begin{osservazione} \label{act}
  	We have to define the $T$-action on our product $\fie{A}(r)\times \widetilde{\fie{A}}(N-2r)$:
  	\begin{itemize}
  		\item[i)] if $g$ is even, the action is described by:
  		$$(t_0,t_1).p(x_0,x_1) :=(t_0t_1)^{g/2}p(x_0/t_0,x_1/t_1)$$ for every $(t_0,t_1) \in T$, for every $p \in \fie{A}(r)$;
  		$$(t_0,t_1).(q(x_0,x_1),s):= (q(x_0/t_0,x_1/t_1),t_1^{r-g-1}s)$$ for every $(t_0,t_1)\in T$, for every $(q,s) \in \widetilde{\fie{A}}(N-2r);$
  		
  		\item[ii)] if $g$ is odd, the action is described by:
  		$$(\alpha,\rho).p(x_0,x_1) :=\alpha^{-1}\rho^{(g+1)/2}p(x_0/\rho,x_1)$$ for every $(\alpha,\rho) \in T$, for every $p \in \fie{A}(r)$;
  		$$(\alpha,\rho).(q(x_0,x_1),s):= (q(x_0/\rho,x_1),s)$$ for every $(\alpha,\rho)\in T$, for every $(q,s) \in \widetilde{\fie{A}}(N-2r);$
  		
  		a straightforward computation shows that $c_r$ is a $T$-equivariant morphism.
  	\end{itemize}
  \end{osservazione}
  
 To construct the Chow envelope, we need to pass to the projective setting. We need to construct two different morphisms and we will prove that the coproduct is the Chow envelope required. Therefore we consider the action of  $\fie{G}_m \times \fie{G}_m$ on $\fie{A}(r)\times \widetilde{\fie{A}}(N-2r)$ (as always $N:=2g+2$) defined by the product of the two actions:
 $$\lambda\cdot (f_0,\dots,f_r), := (\lambda f_0,\dots,\lambda f_r)$$
  or equivalently $\lambda\cdot f:= \lambda f$	for every $\lambda \in \fie{G}_m$ and $f=(f_0,\dots,f_r)\in \fie{A}(r)$; 
 	$$ \mu \cdot (g_0,\dots,g_{N-2r},s):=(\mu^2g_0,\dots,\mu^2g_{N-2r},\mu s) $$
 	or equivalently $\mu \cdot (g,s):= (\mu^2 g,\mu s)$ for every $\mu \in \fie{G}_m$ and $(g,s) \in \widetilde{\fie{A}}(N-2r)$. We have denoted the quotient stack $[\widetilde{\fie{A}}(N-2r)\setminus 0/\fie{G}_m]$ for the action described above by 
 	$\fie{P}(2^{N-2r},1)$.
Clearly this action commutes with the one of the torus $T$ described in \hyperref[act]{Remark \ref{act}}. The morphism $c_r$ is equivariant for the morphism of group schemes $\fie{G}_m \times \fie{G}_m \times T \rightarrow \fie{G}_m \times T$ described as $$(\lambda,\mu,t) \mapsto (\lambda\mu,t),$$ 
i.e. $c_r((\lambda,\mu,t)\cdot (f,(g,s)))= (\lambda\mu,t)\cdot c_r(f,(g,s))$. 
We denote by $$a_r:\fie{P}^r \times \fie{P}(2^{N-2r},1) \rightarrow \fie{P}(2^N,1)$$ the morphism induced by $c_r$ on the quotients of $(\fie{A}(r)\setminus 0) \times (\widetilde{\fie{A}}(N-2r)\setminus 0)$ by $\fie{G}_m\times \fie{G}_m$ and of $\widetilde{\fie{A}}(N)\setminus0$ by $\fie{G}_m$. 

\begin{osservazione}
	Unfortunately, the morphism $a_r$ is not a Chow envelope of $\overline{\Delta}_r$. Consider in fact an object $(h,0) \in \overline{\Delta}_r(k) \subset \fie{P}(2^N,1)(k)$ such that $h=f^2g$ where $f,g \in k[x_0,x_1]$ are homogenous polynomials of degree $r, N-2r$ respectively with $g(0,1) \in k\setminus k^2$ and $f(0,1)=0$. Thus, if we hope to find an element $(g,s)$ in $\fie{P}(2^{N-2r},1)$ such that $g(0,1)=s^2$, we need to pass to an extension of $k$. Therefore it will be a surjective morphism of algebraic stacks, but not a Chow envelope.
\end{osservazione}
We need then to construct another morphism to have a Chow envelope in the case $f(0,1)=0$. The construction is the following. Consider the closed immersion $i_r:\fie{A}(r-1) \hookrightarrow \fie{A}(r)$ defined by $$ (f_0,\dots,f_{r-1}) \mapsto (f_0,\dots,f_{r-1},0)$$ and we let $U_r$ be the open complement of this closed subset. Notice that the closed immersion can be expressed in the language of homogenous polynomials as $f \mapsto x_0f$, if $f \in \fie{A}(r-1)$. 

We consider the morphism 
$$ d_r: \fie{A}(r-1) \times \fie{A}(N-2r) \rightarrow \widetilde{\fie{A}}(N)$$ defined by the equation 
\begin{equation*}\label{d_r}
	(f,g) \mapsto ((x_0f)^2g,0) \tag{**}
\end{equation*}
and a $\fie{G}_m$-action on $\fie{A}(N-2r)$ described by 
$$ \mu \cdot (g_0,\dots,g_{N-2r}):=(\mu^2g_0,\dots,\mu^2g_{N-2r})$$ 
or equivalently $\mu\cdot g:=\mu^2 g$ for $\mu \in \fie{G}_m$ and $g \in \fie{A}(N-2r)$. We will denote by $\fie{P}(2^{N-2r+1})$ the quotient stack $[\fie{A}(N-2r) \setminus 0/\fie{G}_m]$ with this action. In the same way, $d_r$ is equivariant for the group scheme homomorphism $\fie{G}_m\times \fie{G}_m \times T \rightarrow \fie{G}_m \times T$ and therefore gives us the morphism $b_r:\fie{P}^{r-1} \times \fie{P}(2^{N-2r+1}) \rightarrow \fie{P}(2^N,1)$. In this case the action of $T$ on $\fie{A}(N-2r)$ is just the restriction to the diagonal torus $T$ of the standard action $A\cdot f(x):= f(A^{-1}x)$ for every $A\in {\rm GL}_2$ and $f\in \fie{A}(N-2r)$.
\begin{lemma}\label{b}
In the setting above, the two maps 
	$$b_r: \fie{P}^{r-1} \times \fie{P}(2^{N-2r+1}) \longrightarrow \fie{P}(2^N,1)$$ 
and 
	  $$a_r: \fie{P}^r \times \fie{P}(2^{N-2r},1)  \longrightarrow \fie{P}(2^N,1) $$ 
are representable proper morphisms of quotient stacks.
\end{lemma}

\begin{proof}
  Representability follows directly from the following fact (see \cite[Lemma 99.6.2]{stacks-project}): suppose $G$ and $H$ are two group schemes with a group homomorphism $\phi:G \rightarrow H$ and suppose $X$ is a $G$-scheme, $Y$ is a $H$-scheme and $f: X \rightarrow Y$ is an equivariant morphism, i.e. $f(g\cdot x)=\phi(g)\cdot f(x)$ for every $x \in X(S)$ and $g\in G(S)$. Then the induced morphism $[X/G]\rightarrow [Y/H]$ between the quotient stacks is representable if and only if for every $k$-scheme $S$ and for every $x \in X(S)$ the following map
	$$ {\rm Stab}_G(x) \rightarrow {\rm Stab}_H(f(x)) $$ is injective, where ${\rm Stab}_G(x):=\{g\in G(S)\vert g\cdot x=x\}$. 
	
	Properness follows from the fact that the source of the morphism is a proper stack and the target is separated (see \cite[Proposition 10.1.6]{Oll})
\end{proof}

\begin{osservazione}\label{dia}
	If we consider now the commutative diagram 
	$$\xymatrix{  \fie{A}(r) \times \widetilde{\fie{A}}(N-2r) \ar[r]^<<<<{c_r} \ar[d]^{{\rm Id}\times \varphi_{N-2r}} &  \widetilde{\fie{A}}(N) \ar[d]^{\varphi_N} \\ 
	              \fie{A}(r) \times \fie{A}(N-2r) \ar[r]^<<<<{\pi_r} & \fie{A}(N)},$$
  we can pass to the projective setting to get the following commutative diagram
  $$\xymatrix{  \fie{P}^r \times \fie{P}(2^{N-2r},1) \ar[r]^<<<<{a_r} \ar[d]^{{\rm Id}\times \phi_{N-2r}} &  \fie{P}(2^N,1) \ar[d]^{\phi_N} \\ 
  	\fie{P}^r \times \fie{P}^{N-2r} \ar[r]^<<<<<<{\pi_r} & \fie{P}^{N}}$$
  which shows that $a_r$ factors through the closed immersion $\overline{\Delta}_r \subset \fie{P}(2^N,1)$. In a similar way, the same can be shown for $b_r$. 
   \end{osservazione}

\begin{lemma}\label{lem}
	Consider the morphisms
	$$b_r: \fie{P}^{r-1} \times \fie{P}(2^{N-2r+1}) \longrightarrow \overline{\Delta}_r $$ 
	and 
	$$a_r: \fie{P}^r \times \fie{P}(2^{N-2r},1)  \longrightarrow \overline{\Delta}_r $$ 
	and let $\omega_r$ be the coproduct morphism. If ${\rm char}(k)>2g$, then $\omega_r$ restricted to the preimage of $\overline{\Delta}_r\setminus\overline{\Delta}_{r+1}$ is a Chow envelope for every $1\leq r \leq N/2$ (where $\overline{\Delta}_{N/2+1}:=\emptyset$).
\end{lemma}

\begin{proof}
  Let us denote $\overline{\Delta}_r\setminus\overline{\Delta}_{r+1}$ by $D$. Consider $K$ an extension of $k$ and $(h,t)\in D(K)$, thus $h \in \Delta_r\setminus \Delta_{r+1}$ and therefore $h=f^2g$ with $f,g \in K[x_0,x_1]$ homogeneous polynomials, where $g$ is square free and deg$f=r$ (see \cite[Lemma 3.2]{Vis3}). Moreover, if  $f(0,1)\neq 0$ the equation
  $$ t^2=h(0,1)=f(0,1)^2g(0,1)$$
  gives us that $a_r(K)(f,(g,s))=(h,t)$ for $s=h(0,1)/f(0,1)$ (we need that $s^2=g(0,1)$). On the other hand, if $f(0,1)=0$ ($f=x_0f'$) then $t=0$ therefore we can consider $(f',g)$ as an element of $( \fie{P}^{r-1} \times \fie{P}(2^{N-2r+1}))(K)$ and get 
  $b_r(f',g)=(h,0)$. We have then proved the statement.
\end{proof}
\begin{comment}

\begin{osservazione}
	Suppose we have $G:=\fie{G}_m$ and $T$ a torus, a scheme $X$ with an action of $G\times G \times T$ and a scheme $Y$ with an action of $G \times T$. Furthermore, we consider a morphism $X \rightarrow Y$ equivariant in respect to the map $G\times G \times T \rightarrow G \times T$ defined as $(g,h,t)\mapsto(gh,t)$. Suppose that if we consider the action of $G$ on $X$ given by 
	$g\cdot x:=(g,g^{-1})\cdot x$, we obtain a free action with an algebraic space $Z$ as a quotient, and considering the action of $G$ on $Z$ 
\end{osservazione}
\end{comment}
\begin{proposizione}
	Suppose ${\rm char}(k) > 2g$. The morphism 
	$$ \omega := \bigsqcup_{r=1}^{g+1} \omega_r $$
	 is surjective at the level of $T$-equivariant Chow ring, i.e. $$\omega_*: \bigoplus_{r=1}^{g+1}{\rm CH}_T\big(\fie{P}^{r-1}\times \fie{P}(2^{N-2r+1})\big) \oplus {\rm CH}_T\big(\fie{P}^r \times \fie{P}(2^{N-2r},1)\big) \rightarrow {\rm CH}_T(\overline{\Delta})$$
	  is surjective.
\end{proposizione}

\begin{proof}
	First, notice that \hyperref[ChEnv]{Corollary \ref{ChEnv}} states that it is enough to prove that $\omega$ is a Chow envelope. Consider the stratification 
	
	$$ 0\subset \overline{\Delta}_{N/2} \subset \dots \subset \overline{\Delta}_1=\overline{\Delta},$$ 
	\hyperref[lem]{Lemma \ref{lem}} states that $\omega_r$ is a Chow envelope restricted to $\overline{\Delta}_r \setminus \overline{\Delta}_{r+1}$. Therefore the coproduct is a Chow envelope for $\overline{\Delta}$, proving the statement.
\end{proof}
Therefore, we just need to describe the image of the morphisms $(b_r)_*$ and $(a_r)_*$ inside ${\rm CH}_T(\fie{P}(2^N,1))$.
\section{Description of the Chow ring of $\st{H}_{g,1}$}\label{de}
Finally, we can explicitly compute the relations in ${\rm CH}_T(\fie{P}(2^N,1))$ using the pushforward along $a_r$ and $b_r$ in $\fie{P}(2^N,1)$ (these two morphisms are the ones induced on the weighted projective stacks by  $c_r$ and $d_r$, whose description is explicited in the equations \hyperref[c_r]{(*)} and \hyperref[d_r]{(**)} respectively). Our goal is to prove that every relation for $\st{H}_{g,1}$ is in the image of the map $\varphi_{2g+2}^*: {\rm CH}_T(\fie{A}(2g+2)) \rightarrow {\rm CH}_T(\widetilde{\fie{A}}(2g+2))$.
\begin{osservazione}
Let us consider the two Chow groups computed using \hyperref[c]{Proposition \ref{c}}:
\begin{itemize}
	\item[i)] $$ {\rm CH}_T(\fie{P}^r\times \fie{P}(2^{N-2r},1)) \simeq \frac{\fie{Z}[T_0,T_1,u_1,v_1]}{(p_1(u_1,T_0,T_1),q_1(v_1,T_0,T_1))} $$ 
where $u_1=c_1(\sh{O}_{\fie{P}^r}(1))$, $v_1=c_1(\sh{O}_{\fie{P}(2^{N-2r},1)}(1))$  and $p_1$  is a monic polynomial in the variable $u_1$ of degree $r+1$;
\item[ii)] $$ {\rm CH}_T(\fie{P}^{r-1}\times \fie{P}(2^{N-2r+1})) \simeq \frac{\fie{Z}[T_0,T_1,u_2,v_2]}{(p_2(u_2,T_0,T_1),q_2(v_2,T_0,T_1))} $$ 
where $u_2=c_1(\sh{O}_{\fie{P}^{r-1}}(1))$, $v_2=c_1(\sh{O}_{\fie{P}(2^{N-2r+1})}(1))$   and $p_2$  is a monic polynomial in the variable $u_2$ of degree $r$.
\end{itemize}

Notice that because $c_r$ (respectively $d_r$) is $T$-equivariant, the image of the pushforward of $a_r$ (respectively $b_r$) will be the ideal generated by the pushforwards of the elements of the form $u_1^iv_1^j$ (respectively $u_2^iv_2^j$) where $i\leq r$ (respectively $i\leq r-1$). 
Recall the description of the $T$-equivariant Chow ring of $\fie{P}(2^N,1)$:
	$$ {\rm CH}_T(\fie{P}(2^N,1)) \simeq \frac{\fie{Z}[T_0,T_1,t]}{(P(t,T_0,T_1))}. $$  
\end{osservazione}
 
\begin{lemma}
	Using the notation above, we have the following equations:
$$ a_r^*(t)= u_1+v_1 $$
and 
$$ b_r^*(t) = u_2 + v_2.$$

\end{lemma} 

\begin{proof}
	Let us show the formula for $a_r$. We recall the morphism 
	$$c_r: (\fie{A}(r) \setminus 0) \times (\widetilde{\fie{A}}(N-2r) \setminus 0) \longrightarrow \widetilde{\fie{A}}(N) \setminus 0 $$
	defined by the formula (see \hyperref[c_r]{(*)})
	$$ c_r(f,(g,s))=(f^2g, f(0,1)s) $$
	and we consider the action of $\fie{G}_m$ on the product $(\fie{A}(r) \setminus 0) \times (\widetilde{\fie{A}}(N-2r) \setminus 0)$ defined by the formula 
	$$ \lambda.(f,(g,s))= (\lambda f, (\lambda^{-2}g, \lambda^{-1} s)) $$
	for every $\lambda \in \fie{G}_m$.
	A straightforward computation shows that $c_r(\lambda.(f,(g,s)))=c_r(f,(g,s))$
	therefore we have an induced $T$-equivariant morphism of stacks (in fact it is a morphism of schemes; see \cite[Proposition 23]{EdGra})
    $$ [c_r]:\frac{(\fie{A}(r) \setminus 0) \times (\widetilde{\fie{A}}(N-2r) \setminus 0)}{\fie{G}_m} \longrightarrow \widetilde{\fie{A}}(N) \setminus 0$$
    which fits into the following commutative diagram:
    $$ \xymatrix{ \frac{(\fie{A}(r) \setminus 0) \times (\widetilde{\fie{A}}(N-2r) \setminus 0)}{\fie{G}_m} \ar[rr]^{[c_r]} \ar[d] && \widetilde{\fie{A}}(N) \setminus 0  \ar[d] \\ 
                  \fie{P}^r \times \fie{P}(2^{N-2r},1) \ar[rr]^{a_r} && \fie{P}(2^N,1) } $$
	where the vertical maps are the natural projection maps (clearly the left one is still defined after quotienting by $\fie{G}_m$).
We leave as an  easy verification to the reader that the vertical map on the left is in fact a $\fie{G}_m$-torsor under the action described by the formula
	$$ \mu.[f,(g,s)]= [f, (\mu^2 g, \mu s)].$$ 
	
	Using the fact that the category of $G$-torsors over a fixed stack is in fact a groupoid for every $G$ group scheme, we deduce that the diagram above is in fact cartesian and consequently we get the following equality at the level of first Chern classes 
	     $$ c_1^T(\sh{L}) = a_r^*\big(c_1^T(\sh{O}_{\fie{P}(2^N,1)}(-1))\big) $$
	where $\sh{L}$ is the line bundle associated to the $\fie{G}_m$-torsor
	$$ \frac{(\fie{A}(r) \setminus 0) \times (\widetilde{\fie{A}}(N-2r) \setminus 0)}{\fie{G}_m} \longrightarrow \fie{P}^r \times \fie{P}(2^{N-2r},1); $$
	explicitly, the line bundle can be described as 
	$$ \sh{L}=\frac{(\fie{A}(r) \setminus 0) \times (\widetilde{\fie{A}}(N-2r) \setminus 0) \times \fie{A}^1}{\fie{G}_m \times \fie{G}_m} $$ 
	where the action can be described as follows: $$(\lambda,\mu). (f,(g,s),v)= (\lambda f, (\lambda^{-2}\mu^2 g, \lambda^{-1}\mu s), \mu v)$$
	for every $(\lambda,\mu) \in \fie{G}_m \times \fie{G}_m$ and for every point $(f,(g,s),v)$ in $(\fie{A}(r) \setminus 0) \times (\widetilde{\fie{A}}(N-2r) \setminus 0) \times \fie{A}^1$.
	Using the group isomorphism $\fie{G}_m \times \fie{G}_m \rightarrow \fie{G}_m \times \fie{G}_m$ described by the association $(\lambda, \mu) \mapsto (\lambda, \lambda^{-1} \mu)$, we deduce that $\sh{L}$ is in fact the quotient above with the new action 
	$(\lambda,\mu)(f,(g,s),v)= (\lambda f, (\mu^2 g, \mu s), \lambda\mu v)$ which describes exactly the line bundle whose first $T$-equivariant Chern class is $c_1^T(\sh{O}_{\fie{P}^r}(-1))+c_1^T(\sh{O}_{\fie{P}(2^{N-2r},1)}(-1))$.
	This concludes the proof. The same idea can be used to prove the statement for $b_r$. 
\end{proof}

\begin{osservazione}\label{rem}
	Using projection formula, it is immediate to prove that the ideal we are looking for is generated by the pushforwards through the map $a_r$ of $u_1^i$ (respectively through the map $b_r$ of $u_2^i$)  for every $i\leq r$ (respectively  for $i\leq r-1$).
\end{osservazione}

Let us consider the pull-back of the Chow envelope $\pi_r$ through the morphism of algebraic stacks $\phi_N$, i.e. the following cartesian diagrams
$$ \xymatrix{\st{R}_r \ar[d]_{p} \ar[r]&\st{P}_r \ar[d]^q \ar[r]_{\varpi_r}
	& \fie{P}(2^N,1) \ar[d]_{\phi_N} \\ \fie{P}^{r-1}\times \fie{P}^{N-2r} \ar[r]^{i_r \times {\rm Id}} &
	 \fie{P}^r \times \fie{P}^{N-2r} \ar[r]^>>>>>>>>>{\pi_r} & \fie{P}^N, } $$
 in particular we are defining $\st{P}_r$ and $\st{R}_r$ as fiber products of the diagrams above.
The commutative diagram described in \hyperref[dia]{Remark \ref{dia}}
$$ \xymatrix{ \fie{P}^r \times \fie{P}(2^{N-2r},1) \ar[d]_{{\rm Id} \times \phi_{N-2r}} \ar[r]^<<<<<{a_r}
	& \fie{P}(2^N,1) \ar[d]_{\phi_N} \\
	\fie{P}^r \times \fie{P}^{N-2r} \ar[r]^>>>>>>>>>{\pi_r} & \fie{P}^N } $$
induces a morphism of algebraic stacks $\alpha_r:  \fie{P}^r \times \fie{P}(2^{N-2r},1) \rightarrow \st{P}_r$ such that $\varpi_r \circ \alpha_r = a_r$ and $q \circ \alpha_r = {\rm Id} \times \phi_{N-2r}$.

In the same way, we consider the following commutative  diagram: 
$$ \xymatrix{ \fie{P}^{r-1} \times \fie{P}(2^{N-2r+1}) \ar[d]_{{\rm Id} \times \iota_{N-2r}} \ar[r]^<<<<<{b_r}
	& \fie{P}(2^N,1) \ar[d]_{\phi_N} \\
	\fie{P}^{r-1} \times \fie{P}^{N-2r} \ar[r]^>>>>>>>>>{\pi_r \circ (i_r\times {\rm Id})} & \fie{P}^N } $$
 where $\iota_{N-2r}:\fie{P}(2^{N-2r+1}) \rightarrow \fie{P}^{N-2r}$ is the quotient map induced by the identity on the affine space $\fie{A}(N-2r)\setminus 0$ with the two different actions (on the source we have the action of $\fie{G}_m$ with all weights equal to $2$, on the target all the weights are equal to $1$). Therefore we get the morphism $\beta_r:  \fie{P}^{r-1} \times \fie{P}(2^{N-2r+1}) \rightarrow \st{R}_r$ together with the equalities $\varpi_r\vert_{\st{R}_r} \circ \beta_r = b_r$ and $ p \circ \beta_r = {\rm Id}\times\iota_{N-2r}$.

\begin{proposizione}
	In the setting above, the two morphisms $\beta_r$ and $\alpha_r$ are representable and proper. Furthermore, we get that $\beta_r(K)$ and $\alpha_r\vert_{\alpha_r^{-1}(\st{P}_r \setminus \st{R}_r)}(K)$ are equivalences of groupoids for every extension of fields $K/k$.
\end{proposizione}

\begin{proof}
	Representability and properness follow from the representability and properness of $b_r$ and $a_r$. Let us start with $a_r$ restricted to the preimage of the open $\st{S}_r:=\st{P}_r\setminus \st{R}_r$ inside $\st{P}_r$. An object inside $\st{S}_r(K)$ is a triplet of the form $(f,g,(h,s))$ where $(f,g) \in (\fie{P}^r \times \fie{P}^{N-2r})(K)$ with $f(0,1)\neq 0$, $(h,t)\in \fie{P}(2^N,1)(K)$ with $h=f^2g$. The only morphisms are the identities if $t\neq 0$, otherwise we have ${\rm Hom}_{\st{S}_r}((f,g,(h,0)),(f,g,(h,0)))=\mu_2(K)$ where $\mu_2$ is the group of the square roots of unity. We can describe the morphism $\alpha_r$ in the following way:
	$$ \alpha_r(f,(g,s))=(f,g,a_r(f,(g,s)))=(f,g,(f^2g,f(0,1)s)).$$
	Therefore because ${\rm Hom}_{\fie{P}(2^{N-2r},1)}((g,s),(g,s))$ is the trivial group if $s\neq 0$ and it is $\mu_2(K)$ if $s=0$, the fact that $f(0,1)\neq 0 $ implies that $\alpha_r$ is fully faithful (it is faithful because of representability). As far as essential surjectivity is concerned, the idea is the same of \hyperref[lem]{Lemma \ref{lem}}. Consider an element $(f,g,(f^2g,t)) \in \st{S}_r(K)$, we get $t^2=f(0,1)^2g(0,1)$, therefore if we define $s:=t/f(0,1)$ we get $\alpha_r(f,(g,s))=(f,g,(f^2g,t))$. We have proved that $\alpha_r(K)\vert_{\alpha_r^{-1}(\st{S}_r)}$ is essentially surjective, therefore equivalence.
	As far as $\beta_r$ is concerned, the proof works in the same way. Recall that the map $i_r: \fie{P}^{r-1} \hookrightarrow \fie{P}^r$  can be described as $f\mapsto x_0f$. Thus we have the following description of $\beta_r$:
	$$ \beta_r(K)(f,g)=(f,g,((x_0f)^2g,0))$$ 
	for every $(f,g)\in (\fie{P}^{r-1} \times \fie{P}(2^{N-2r+1}))(K)$ and again fully faithfulness and essential surjectivity are straightforward.
\end{proof}

\begin{lemma}\label{us}
	In the setting of previous proposition, we get the following equalities at the level of Chow rings: 
	$(\beta_r)_*(1)=1$ and $(\alpha_r)_*(1)=1.$ 
\end{lemma}
\begin{proof}
	The statement follows easily from the fact that the property of being an equivalence on points is stable under base change, thus we can pass to an approximation (see proof of \hyperref[lem2]{Lemma \ref{lem2}}) and we just need to prove the following statement: given a proper morphism  $\alpha:X \rightarrow Y$ of algebraic spaces over $k$ such that $\alpha(K)$ is bijective for every field extension $K/k$, then $\alpha_*(1)=1$ at the level of Chow group. This follows from the fact that the hypothesis of the claim implies $\alpha$ is a birational morphism.
\end{proof}

\begin{osservazione}
	Before stating and proving the theorem, we recall that for every flat morphism of Deligne-Mumford separated stacks we have an induced pullback morphism at the level of Chow ring (we do not need representability) and the same is true in the $T$-equivariant case. Moreover, the compatibility formula applies between flat pullbacks and proper representable pushforward (see \cite[Lemma 3.9]{Vis1}).
\end{osservazione}

We are finally ready to prove our main theorem.

\begin{teorema}\label{pri}
	Let $g\geq 2$ be an integer, and ${\rm char}(k)>2g$. Then every relation coming from $\overline{\Delta}$ inside $\fie{P}(2^{2g+2},1)$ is the image through the morphism $\phi_{2g+2}^*$ of a relation coming from $\Delta$ inside $\fie{P}^{2g+2}$.
	 We denote the ideal inside ${\rm CH}({\rm B}{\rm GL}_2/\mu_{g+1})$ defining the relations for $\st{H}_g$ by $J$  and the ideal generated by the image of $J$ through the map $\varphi_{2g+2}^*$ by $I$. Then we get the following isomorphism  
	$$ {\rm CH}(\st{H}_{g,1}) \simeq \frac{ {\rm CH}({\rm B}T)}{I}.$$
	
\end{teorema}

\begin{proof}
	
	The first statement implies the remaining part of the theorem by using the surjectivity of the pullback in the case of a $\fie{G}_m$-torsor quotient map (see \hyperref[GM]{Proposition \ref{GM}}) and diagram chasing. Therefore we just need to prove that the image of the pushforward of $b_r$ and $a_r$ are contained inside the image of $\phi_{2g+2}^*$. As usual, we set $N:=2g+2$. By \hyperref[rem]{Remark \ref{rem}} it is enough to prove that $(a_r)_*(u_1^i)$ (respectively $(b_r)_*(u_2^i)$) is the image of some element through $\phi_N^*$ for every $i\leq r$ (respectively for every $i\leq r-1$).
	The proof is the same for both $a_r$ and $b_r$. Let us deal with $a_r$. We have the following chain of equalities, for every $i\leq r$: 
	\begin{eqnarray}
	(a_r)_*(u_1^i)  = (\varpi_r)_*(\alpha_r)_*(u_1^i) = (\varpi_r)_*(\alpha_r)_*({\rm Id}\times \phi_{N-2r})^*(c_1(\sh{O}_{\fie{P}^r}(1)\boxtimes \sh{O}_{\fie{P}^{N-2r}})^i) \nonumber \\  =   (\varpi_r)_*(\alpha_r)_*\alpha_r^*q^*(c_1(\sh{O}_{\fie{P}^r}(1)\boxtimes \sh{O}_{\fie{P}^{N-2r}})^i) \nonumber.
	\end{eqnarray} 
	By abuse of notation, the element $c_1(\sh{O}_{\fie{P}^r}(1)\boxtimes \sh{O}_{\fie{P}^{N-2r}})$ will be denoted by $u_1$. Using projection formula and \hyperref[us]{Lemma \ref{us}} we get 
	$$ (\alpha_r)_*(\alpha_r)^*(q^*(u_1^i))= q^*(u_1^i) $$ 
	because $u_1$ is the first Chern class of a line bundle.
	Therefore, we get the following equality: 
	$$ (a_r)_*(u_1^i)= (\varpi_r)_* q^* (u_1^i) = \phi_N^*(\pi_r)_*(u_1^i) $$
	which concludes the proof.
\end{proof}

Using the theorem above and the description of the Chow ring of $\st{H}_g$ obtained in \cite[Theorem 1.1]{EdFul} when $g$ is even and \cite[Theorem 6.2]{DiLor} when $g$ is odd, we get the following two descriptions of the Chow ring of the stack of pointed hyperelliptic curves.
\begin{teorema}\label{p}
	If $g$ is even, $g\geq 2$ and ${\rm char}(k)>2g$, then we have the following isomorphism
	$$ {\rm CH}(\st{H}_{g,1})\simeq \frac{\fie{Z}[T_0,T_1]}{\bigg(2(2g+1)(T_0+T_1),g(g-1)(T_0^2+T_1^2)-2g(g+3)T_0T_1\bigg)}.$$ 
	
\end{teorema}

\begin{teorema}\label{d}
		If $g$ is odd, $g\geq 3$ and ${\rm char}(k)>2g$, then we have the following isomorphism
	$$ {\rm CH}(\st{H}_{g,1})\simeq \frac{\fie{Z}[\tau,\rho]}{\bigg(4(2g+1)\tau,8\tau^2+2g(g+1)\rho^2\bigg)}.$$ 
\end{teorema}

In the case of $g$ even, the theorem is a conseguence of making explicit the map
$$\phi_N^*:{\rm CH}_{{\rm GL}_2}(\fie{A}(N))  \rightarrow {\rm CH}_T(\widetilde{\fie{A}}(N)); $$
the morphism is clearly induced by 
$$ {\rm CH}({\rm BGL}_2)\simeq \fie{Z}[c_1,c_2] \rightarrow \fie{Z}[T_0,T_1]\simeq {\rm CH}({\rm BT})$$
where $c_1 \mapsto T_0+T_1$ and $c_2 \mapsto T_0T_1$, because it is the pullback of the diagonal inclusion $T=(\fie{G}_m)^2 \hookrightarrow {\rm GL}_2$.

In the case $g$ odd, we have to understand the morphism 
$$\phi_N^*:{\rm CH}_{\fie{P}{\rm GL}_2\times \fie{G}_m}(\fie{A}(N))  \rightarrow {\rm CH}_T(\widetilde{\fie{A}}(N)); $$
this morphism too is induced by 
$$ {\rm CH}({\rm B}\fie{P}{\rm GL}_2 \times {\rm B}\fie{G}_m)\simeq \frac{\fie{Z}[\tau,c_1,c_2,c_3]}{(c_1,2c_3)} \rightarrow \fie{Z}[\tau,\rho]\simeq {\rm CH}({\rm BT});$$
which is the pullback of the group homomorphism 
$$ \fie{G}_m \times \fie{G}_m \rightarrow \fie{G}_m \times\fie{P}{\rm GL}_2$$
given by $(t, l) \mapsto (t, A)$ where $A$ is the class in $\fie{P}{\rm GL}_2$ given by the matrix 

$$A(l)=	\begin{bmatrix}
l & 0 \\
0 & 1
\end{bmatrix} \in  \fie{P}{\rm GL}_2.$$
We notice that $\tau$ in the Chow group is the generator of ${\rm B}\fie{G}_m$ inside the product, thus we get $\phi_N^*(\tau)=\tau$. To understand the image of $c_1,c_2,c_3$ we need to explicit where they come from geometrically. Pandariphande shows in \cite{Pan} that they are the Chern classes of the adjunction representation of $\fie{P}{\rm GL}_2$ on its Lie algebra. It is straightforward to see that for every $l\in \fie{G}_m$ the matrix $A(l)$ acts with eigenvalues $l,1,1/l$ and therefore the pullback on the Chow rings is defined by 
\begin{itemize}
	\item $c_1 \mapsto 0;$
	\item $c_2 \mapsto -\rho^2;$
	\item $c_3 \mapsto 0.$
\end{itemize}

Mapping the two relations of $\st{H}_g$ for $g$ odd through this map gives us the statement of \hyperref[d]{Theorem \ref{d}}. Finally, we have the explicit description of the integral Chow ring of the stack of 1-pointed hyperelliptic curves of genus $g$.

\section{Generators of the Chow ring}\label{gen}
In this last part, we will give the geometric interpretation of the generators of the Chow ring of $\st{H}_{g,1}$. We divide it in two cases, depending on the parity of the genus. In the even case, we will prove the following theorem.

\begin{teorema}
	Suppose $g$ is an even positive integer and ${\rm char}(k)>2g$, then in the notation of \hyperref[p]{Theorem \ref{p}} we have
	$ T_0 = c_1(\sh{B}_g) $ and $T_1 = c_1(\sh{A}_g)$ where $\sh{A}_g$ and $\sh{B}_g$ are the two line bundles on $\st{H}_{g,1}$ defined as $$\sh{B}_g(\pi:C\rightarrow S,\sigma):=\sigma^*\omega_{C/S}^{\otimes g/2}((1-g/2)W)$$ and $$\sh{A}_g(\pi:C\rightarrow S,\sigma):=\pi_*\omega_{C/S}^{\otimes g/2}((1-g/2)W-\sigma)$$ with $W$ the Weierstrass divisor associated to a relative hyperelliptic curve $C\rightarrow S$.
\end{teorema}

Let us start with the description in the non-pointed case. Edidin and Fulghesu have given the explicit formula for the generators of the Chow ring of $\st{H}_g$. We recall that the generators are  $c_1(\fie{V}_g)$ and $c_2(\fie{V}_g)$ where the $2$-dimensional vector bundle $\fie{V}_g$ is defined by the following: for every morphism $S\rightarrow  \st{H}_g$ associated to the hyperelliptic curve $\pi:C\rightarrow S$, we have
$$ \fie{V}_g(S)= \pi_*(\omega_{C/S}^{\otimes g/2}((1-g/2)W)) $$ 
where $\omega_{C/S}$ is the relative canonical bundle and $W$ is the ramification divisor. They proved that this is infact the pullback over the natural map $\st{H}_g \rightarrow {\rm BGL}_2$ of the ${\rm GL}_2$-representation $ E \otimes ({\rm det}E)^{\otimes g/2}$ where $E$ is the standard representation of ${\rm GL}_2$ (which is in turn the pullback through the isomorphism ${\rm B}({\rm GL}_2/\mu_{g+1}) \simeq {\rm BGL}_2$
of the standard representation $E$).

\begin{lemma}\label{le}
	Let $\pi:C\rightarrow S$ be a smooth curve over $S$, i.e. a smooth, proper morphism such that every geometric fiber is a connected one dimensional scheme. Suppose $\sh{L}$ is a line bundle on $C$ such that $\sh{L}\vert_{C_s}$ is globally generated for every geometric point $s$ in $S$. Moreover, suppose there exists a section $\sigma: S \rightarrow C$ of the morphism $\pi$. Then the following sequence 
$$ \xymatrix{ 0 \ar[r] & \pi_*\sh{L}(-\sigma) \ar[r] & \pi_*\sh{L} \ar[r] & \sigma^*\sh{L} \ar[r] & 0} $$
is exact.
	
\end{lemma}
\begin{proof}
		We consider the exact sequence induced by $\sigma$
	$$ \xymatrix{ 0 \ar[r] & \sh{O}_C(-\sigma) \ar[r] & \sh{O}_C \ar[r] & \sigma_*\sigma^*\sh{O}_C \ar[r] & 0} $$
	and we tensor it with $\sh{L}$ to obtain 
	$$ \xymatrix{ 0 \ar[r] & \sh{L}(-\sigma) \ar[r] & \sh{L} \ar[r] & \sigma_*\sigma^*\sh{L} \ar[r] & 0}. $$
		Clearly $R^1\pi_*(\sigma_*\sigma^*\sh{L})=0$ because $\sigma_*\sigma^*\sh{L}$ restricted to every geometric fiber is supported on a point. Therefore the natural map 
	$$ \xi: R^1\pi_*\sh{L}(-\sigma) \longrightarrow R^1\pi_*\sh{L}$$
	is surjective. Because $R^2\pi_*\sh{F}=0$ for every coherent sheaf $\sh{F}$, using cohomology and base change we know that $\xi$ restricts  to the morphism 
	$$ \xi_s:H^1(C_s,\sh{L}_s(-\sigma(s))) \longrightarrow H^1(C_s,\sh{L}_s) $$
	for every geometric point $s \in S$. Recall that on a smooth curve $C$ over an algebraically closed field a divisor $D$ is globally generated (without base points) if and only if $$h^1(D-P)=h^1(D)$$ for every $P$ closed point on $C$. Being $\xi_s$ surjective, this clearly implies $\xi_s$ isomorphism for every geometric point and therefore $\xi$ isomorphism.

	Thus, applying $\pi_*$ to the exact sequence
		$$ \xymatrix{ 0 \ar[r] & \sh{L}(-\sigma) \ar[r] & \sh{L} \ar[r] & \sigma_*\sigma^*\sh{L} \ar[r] & 0}$$
	 we get the statement.

\end{proof}
	Let $\phi$ be the natural morphism $\st{H}_{g,1}\rightarrow \st{H}_g$.
	Our idea is to define the two line bundles $\sh{A}_g$ and $\sh{B}_g$, whose first Chern classes generate the Chow ring of $\st{H}_{g,1}$, using the fact that $\varphi^*c(\fie{V}_g)=c(\sh{A}_g)c(\sh{B}_g)$. Fix a morphism $S\rightarrow \st{H}_{g,1}$ induced by the $1$-pointed hyperelliptic curve $(C\rightarrow S,\sigma)$. We define 
	$$ \sh{A}_g:=\pi_*\sh{L}(-\sigma),     \quad \sh{B}_g:=\sigma^{*}\sh{L}$$
	where $\sh{L}:= \omega_{C/S}^{\otimes g/2}((1-g/2)W)$. Notice that $\sh{L}_s$ is the divisor $f^*\sh{O}_{\fie{P}^1_{\overline{k}(s)}}(1)$ where $f:C_s \rightarrow \fie{P}^1_{\overline{k}(s)}$ is the degree $2$ cover of $\fie{P}^1$ of the hyperelliptic curve, therefore $\sh{L}_s$ is globally generated for every geometric point $ s \in S$. As we are in the hypotheses of \hyperref[le]{Lemma \ref{le}} we get that $c(\fie{V}_g)=c(\sh{A}_g)c(\sh{B}_g)$.
	
	\begin{osservazione}
		Recall that a vector bundle over ${\rm BB}_2$, where ${\rm B}_2$ is the Borel subgroup of ${\rm GL}_2$, is equivalent to a flag of ${\rm GL}_2$-representations  $0\subset F \subset E$ of length $2$. Notice that at the level of Chern classes it is the same as supposing that $E \simeq F\oplus E/F$, which is confirmed by the fact that ${\rm B}_2$-equivariant Chow group is isomorphic to the $T$-equivariant one, where $T$ is the maximal torus inside ${\rm GL}_2$.
		Given such a flag, we will say that $c(F)$ and $c(E/F)$ are the total Chern classes induced by it.
		\end{osservazione}
	
	In our situation, $T_0$ and $T_1$ will be the first Chern classes induced by the flag 
	$$0\subset F\otimes ({\rm det}E)^{\otimes g/2} \subset E\otimes({\rm det}E)^{\otimes g/2},$$
	 where $E$ is the standard representation of ${\rm GL}_2$, and $F$ is the subrepresentation stable under the standard action of ${\rm B}_2$ (in our setting the Borel subgroup is identified with the lower triangular matrices). 
	
	We consider as usual a $1$-pointed hyperelliptic curve $(\pi:C \rightarrow S,\sigma)$  and $\sh{L}$ is defined by $\sh{L}:=\omega_{C/S}^{\otimes g/2}((1-g/2)W)$. We start by taking the flag of vector bundles of $\st{H}_{g,1}$ induced by \hyperref[le]{Lemma \ref{le}}:
	$$ 0 \subset \pi_*\sh{L}(-\sigma) \subset \pi_*\sh{L} $$
    and let us consider the pullback to the atlas of $\st{H}_{g,1}$ that we denoted by ${\rm H}_{g,1}'$ (see \hyperref[re]{Remark \ref{re}}).
	This is described by taking the pushforward of the flag $0\subset \sh{L}(-\sigma) \subset \sh{L}$ through $\pi:C \rightarrow S$, which is the composition of the two maps $p:\fie{P}^1_S \rightarrow S$ and $f: C \rightarrow \fie{P}^1_S$. If we first take the pushforward by $f$, which is finite and flat of degree 2, we get
	$$ \xymatrix{ 0 \ar[r] & \sh{O}_{\fie{P}_S^1}\oplus \sh{O}_{\fie{P}^1_S}(-g) \ar[r] & \sh{O}_{\fie{P}^1_S}(1)\oplus \sh{O}_{\fie{P}^1_S}(-g) \ar[r] & (\sigma_{\infty})_*\sigma_{\infty}^*\sh{O}_{\fie{P}^1_S}(1) \ar[r] & 0.} $$
	Finally taking the pushfoward through $p$ we get the pullback to $S$ of the standard flag $0\subset F\subset E$ of ${\rm B}_2$ (up to the isomorphism ${\rm B}({\rm B}_2/\mu_{g+1}) \simeq {\rm BB}_2$). Notice that $p_*\sh{O}_{\fie{P}^1_S}(-g)=0$.
	If we compare it to the presentation exhibited in \hyperref[p]{Corollary \ref{p}}, we have proved the following description of the generators
	$$ T_0=c_1(\sh{A}_g)    \quad T_1=c_1(\sh{B}_g)$$ 
	where $\sh{A}_g$ and $\sh{B}_g$ are the two line bundles described above. 
	
	Now we pass to the odd case. The result is the following.
	
	\begin{teorema}
		Suppose $g$ is an odd positive integer with $g>1$ and ${\rm char}(k) > 2g$, in the notation of \hyperref[d]{Theorem \ref{d}}, we have $\tau=c_1(\sh{L})$ and $\rho=c_1(\sh{R})$ where $$\sh{L}(\pi:C\rightarrow S,\sigma):=\pi_*\omega_{C/S}^{\otimes \frac{g+1}{2}}\Big(\frac{1-g}{2}W\Big)$$
		and $\sh{R}(\pi:C \rightarrow S,\sigma):=\pi_*\omega_{C/S}^{-1}(W-2\sigma)$, with $W$ the Weierstrass divisor associated to an hyperelliptic curve $C\rightarrow S$.
	\end{teorema}
	
	If $g$ is odd, we can apply the same reasoning, using the geometric description of the two generators of ${\rm CH}(\st{H}_g)$ given in \cite{DiLor}. 
    We consider the two generators $\tau$ and $\rho$ like in \hyperref[d]{Corollary \ref{d}}. Clearly by construction $\tau$ is the same compared to the one in \cite{DiLor};
    therefore $\tau=c_1(\sh{L})$ where $\sh{L}$ is functorially defined by 
    $$ \sh{L}((\pi:C\rightarrow S),\sigma) = \pi_*\omega_{C/S}^{\otimes \frac{g+1}{2}}\Big(\frac{1-g}{2}W\Big)$$
    where $W$ is the ramification divisor. The other generators in the case of $\st{H}_g$ are the Chern classes of the $3$-dimensional vector bundle $\sh{E}$ defined functorially by 
    $$\sh{E}(\pi:C\rightarrow S)= \pi_*\omega_{C/S}^{-1}(W).$$
    We repeat the procedure used in the case of $g$ even. If we consider a morphism $ S \rightarrow \st{H}_{g,1}$ induced by the element $(\pi:C\rightarrow S,\sigma)$, we can consider the flag on $S$ 
    $$ 0 \subset \pi_*\omega_{C/S}^{-1}(W)(-2\sigma) \subset \pi_*\omega_{C/S}^{-1}(W)(-\sigma) \subset \pi_*\omega_{C/S}^{-1}(W) $$ 
    that corresponds to the flag of ${\rm B_3}$-representations 
    $$ 0\subset \sh{O}_S \subset \sh{O}_S^{\oplus 2} \subset \sh{O}_S^{\oplus 3} $$ 
    induced by the inclusions $f \mapsto (0,f)$ and $(g,h) \mapsto (0,g,h)$. Using the adjunction representation map $Ad: \fie{G}_m\subset \rm{PB}_2 \rightarrow {\rm B}_3$ we easily get that $\rho=c_1(\sh{R})$ where $\sh{R}$ is functorially defined by 
    $$ \sh{R}(\pi:C\rightarrow S,\sigma):=\pi_*\omega_{C/S}^{-1}(W-2\sigma).$$
    
    We have thus given the geometric interpretion of the generators of ${\rm CH}(\st{H}_{g,1})$ in both the cases of $g$ even and odd.

\bibliographystyle{alpha}
%\bibliography{Bibliografia}

 \end{document}